\documentclass[letterpaper, 10pt, conference]{ieeeconf}      

\usepackage{umlaute,eurosym,comment}
\usepackage{amsmath,amssymb,bm,upgreek,theorem}
\usepackage{float,epsfig,pstricks,pst-node,subfig,xcolor}
\usepackage{wrapfig}
\usepackage{ifpdf,cite}
\usepackage{multirow}
\usepackage{flushend}
\usepackage{tikz,pgfplots}

\psset{unit=1mm,framesep=10pt,fillstyle=none}  
\degrees[360]                                  
\psset{linewidth=0.9pt}                        
\providecommand{\stl}{\psset{linewidth=0.9pt}} 

\newlength\figureheight	
\newlength\figurewidth	

\newtheorem{theorem}{Theorem}

\newtheorem{lemma}[theorem]{Lemma}

\newtheorem{proposition}[theorem]{Proposition}

\newtheorem{remark}[theorem]{Remark}
\newtheorem{assumption}[theorem]{Assumption}

\providecommand{\ef}{\;.}
\providecommand{\ec}{\;,}

\providecommand{\fa}{\forall\;}         
\providecommand{\eqdef}{\triangleq}     

\newcommand{\BN}{\mathbb{N}}        
\newcommand{\BZ}{\mathbb{Z}}        
\newcommand{\BR}{\mathbb{R}}        

\definecolor{ULOcean}{RGB}{0,75,90}
\definecolor{ULOrange}{RGB}{236,116,4}
\definecolor{LiteRed}{RGB}{255,206,206}
\definecolor{MedRed}{RGB}{204,24,24}
\definecolor{DarkRed}{RGB}{127,15,15}
\definecolor{LiteYellow}{RGB}{255,255,100}
\definecolor{MedYellow}{RGB}{255,255,25}
\definecolor{DarkYellow}{RGB}{225,225,0}
\definecolor{LiteBlue}{RGB}{170,184,255}
\definecolor{MedBlue}{RGB}{10,20,204}
\definecolor{DarkBlue}{RGB}{27,29,120}
\definecolor{LiteGreen}{RGB}{24,204,75}
\definecolor{MedGreen}{RGB}{0,130,37}
\definecolor{DarkGreen}{RGB}{0,61,17}
\definecolor{LiteViolet}{RGB}{110,0,225}
\definecolor{MedViolet}{RGB}{73,0,148}
\definecolor{DarkViolet}{RGB}{53,0,107}
\definecolor{LiteOrange}{RGB}{255,226,117}
\definecolor{MedOrange}{RGB}{255,206,0}
\definecolor{DarkOrange}{RGB}{204,167,10}
\definecolor{LiteGray}{RGB}{230,230,230}
\definecolor{MedGray}{RGB}{160,160,160}
\definecolor{DarkGray}{RGB}{100,100,100}

\title{\LARGE \bf

Contingency Model-based Control (CMC) for Communicationless Cooperative Collision Avoidance in Robot Swarms 
}

\author{Georg Schildbach$^{*}$
\thanks{$^{*}$Autonomous Systems Laboratory, University of L{\"u}beck, L{\"u}beck, Germany (email: {\tt\small georg.schildbach@uni-luebeck.de}).}
}

\begin{document}

\maketitle
\thispagestyle{empty}
\pagestyle{empty}


\begin{abstract}
  Cooperative collision avoidance between robots, or `agents,' in swarm operations remains an open challenge. Assuming a decentralized architecture, each agent is responsible for making its own decisions and choosing its control actions. Most existing approaches rely on a (wireless) communication network between (some of) the agents. In reality, however, communication is brittle. It may be affected by latency, further delays and packet losses, and transmission faults. Moreover, it is subject to adversarial attacks, such as jamming or spoofing. This paper proposes Contingency Model-based Control (CMC), a decentralized cooperative approach that does not rely on communication. Instead, the control algorithm is based on consensual rules that are designed for all agents offline, similar to traffic rules. For CMC, this includes the definition of a contingency trajectory for each robot, and perpendicular bisecting planes as collision avoidance constraints. The setup permits a full guarantee of recursive feasibility and collision avoidance between all swarm members in closed-loop operation. CMC naturally satisfies the plug \& play paradigm, i.e., new robots may enter the swarm dynamically. The effectiveness of the CMC regime is demonstrated in two numerical examples, showing that the collision avoidance guarantee is intact and the robot swarm operates smoothly in a constrained environment.
\end{abstract}

\vspace*{0.3cm}
\section{INTRODUCTION}\label{Sec:Intro}

Swarm robotics is inspired by the behavior of biological systems, such as animals operating in swarms, flocks, or packs. Similarly, robots are expected to collaborate in technical applications, using simple, local interaction rules \cite{Hamann:2018}. A robot swarm may thus become more capable, adaptive, (cost) efficient, and resilient than a single robot. However, the control of the individual robots, which includes motion planning in the context of this paper, represents a major challenge. One of the most critical aspects is collision avoidance between the robots, in particular, if the robots operate in dense environments and/or at high speeds.

\subsection{Centralized Approaches} 

Centralized approaches collect relevant information from all the robots, or `agents,' and issue individual provisions on their control actions. The global optimal control problem can be formulated and solved, e.g., based on graph search \cite{SvestkaOvermars:1998,PreissEtAl:2017}, mixed integer programming \cite{SchouwenaarsEtAl:2001}, sequential convex programming \cite{AugugliaroEtAl:2012}, or dynamic programming \cite{StipanovicEtAl:2007}. The key advantage of centralized approaches is the performance level that may be achieved if a centralized decision maker weights the individual preferences for the sake of a collective goal of the swarm. The main drawback is the requirement of a centralized decision maker, which must be alive and accessible by all agents. Moreover, it must possess sufficient computational power in order to solve the optimal control problem for all agents.

\subsection{Decentralized Approaches} 

Decentralized approaches eliminate the need for a central decision maker. The task of controlling the swarm is distributed between the agents, each of them solving an individual (optimal) control problem. Different schemes have been proposed for the iterations between these individual control problems. They may be distinguished into sequential, iterative, and synchronous approaches \cite{VanParysPipeleers:2017}. A common feature is that collision avoidance is achieved by communication between (some of) the agents, exchanging their planned trajectories \cite{ChalousEtAl:2010,ChengEtAl:2017,LuisSchoellig:2019,GraefeEtAl:2025}, dual variables \cite{OngGerdes:2015,SaccaniFagiano:2021,CarronEtAl:2023}, or pairwise constraints \cite{BhattacharyaEtAl:2011,VanParysPipeleers:2017}. 

In exchange for a distribution of the computational effort, decentralized approaches feature an increased amount of communication compared to centralized approaches. To reduce the communication volume, additional constraints or restricted action sets may be introduced for the agents, e.g., using safe corridors \cite{ToumiehFloreano:2024} or polyhedral hulls for trajectory intervals \cite{TordesillasHow:2021}. However, this may lead to a sub-optimal performance, often termed as `conservatism.' In consequence, decentralized approaches may exhibit a poor scaling with the number of robots and/or a loss of formal safety guarantees. A further difficulty may arise from dynamic changes in the swarm configuration, i.e., a robot entering (or leaving) the swarm during its operation. This feature, also known as plug \& play control \cite{Stoustrup:2009}, has been specifically addressed in recent work \cite{SaccaniFagiano:2021,CarronEtAl:2023}. 

\subsection{Decentralized Approaches without Communication} 

The key problem of decentralized approaches, as discussed above, remains their reliance on a (wireless) communication network. In fact, communication is always brittle in a real environment. It is inherently affected by (regular) latency, (irregular) delays and losses of packets, and/or the transmission of faulty or noisy data. Moreover, wireless communication is susceptible to adversarial attacks, such as jamming or spoofing. For many practical applications, it is therefore advantageous if essential features related to safety, such as collision avoidance, work entirely without communication. 

Decentralized approaches without communication thus aim to design the individual control problem in such a way that each robot may rely only on its onboard sensors for outside information. To this end, a predefined set of universal rules is incorporated into the design of the control algorithm. A illustrative analogy is road traffic, where the communication between traffic participants is almost non-existent---with the exception of a few visual and acoustic signals which, however, are not essential at all. Despite the lack of communication, vehicles and other traffic participants are passing each other safely, even closely and at high speeds, because everyone adheres to a universal set of (traffic) rules.

In the context of swarm robotics, this approach has also been termed `implicit cooperation' \cite{AbeYoshiki:2001}. Multiple types of algorithms have been proposed in this context. Potential fields or barrier functions, for example, may be used as a repulsion force between the agents of swarm \cite{ChangEtAl:2003,RezaeeAbdollahi:2014}. However, even if they adapt dynamically, their tuning leads to a trade-off between performance and safety, and there are no formal guarantees for collision avoidance. Another possibility is the use of iterative learning control \cite{ZhuEtAl:2020}. It has been shown to work well on repetitive tasks, and it may also be suitable for new tasks if the setups are sufficiently similar. Clearly, there are also no formal safety guarantees. 

A further class of approaches is based on the concept of velocity obstacles (VOs) \cite{FioriniShiller:1998}. VOs prevent collisions under the assumption of a constant velocity for the other robots. They may be imposed on the entire swarm, by sharing the velocity restrictions equally between all robots \cite{VanDenBergEtAl:2009}. This leads to an approach known as Optimal Reciprocal Collision Avoidance (ORCA). It has been improved to include acceleration-velocity obstacles \cite{VanDenBergEtAl:2011} and generalized to non-holonomic robots \cite{AlonsoMoraEtAl:2013}. ORCA comes with a guarantee of collision avoidance, albeit only for a finite amount of time. Another drawback is that ORCA lacks the integration of motion planning to follow more complex objectives, i.e., other than moving along a straight line. For this purpose, the VOs have been integrated into a motion planning framework using Model Predictive Controller (MPC) \cite{ChengEtAl:2017}. The resulting approach works well in practice \cite{KratkyEtAl:2025}. However, formal guarantees of recursive feasibility and collision avoidance critically depend on information about the future trajectories of the other agents, thus relying again on communication.

\subsection{New Contribution} 

This paper presents a new approach for cooperative collision avoidance in robot swarms, called Contingency Model-based Control (CMC). It is based implicit cooperation paradigm, i.e., it relies on a set of consensual rules and does not require any communication between the agents. It is inspired by the combination of ORCA with MPC \cite{ChengEtAl:2017,KratkyEtAl:2025}. In contrast to the existing approach, the collision avoidance constraints are enforced (only) for a contingency trajectory. In this sense, CMC leverages the idea of multi-trajectory predictions \cite{SaccaniFagiano:2021,CarronEtAl:2023}. The contingency trajectory is completely defined by the current position and speed of each agent. It can thus be constructed by each agent for all other agents using only its onboard sensor measurements. The contingency trajectory is used to set up the mutual collision avoidance constraints based on a perpendicular bisecting plane, tightened by the robot size. It can be shown, mathematically, that this combination permits a guarantee of recursive feasibility and collision avoidance between the robots under the CMC regime in closed loop.



\section{NOTATION}

$\mathbb{R}$ and $\mathbb{Z}$ denote the sets of real and integral numbers. $\mathbb{R}_{+}$ ($\mathbb{R}_{0+}$) and $\mathbb{Z}_{+}$ ($\mathbb{Z}_{0+}$) are the sets of positive (non-negative) real and integral numbers, respectively. For any $r\in\BR$, $\lceil r\rceil$ is the smallest integral number larger than $r$ and $\lfloor r\rfloor$ is the largest integral number lower than $r$.

Bold lowercase and uppercase letters denote vectors and matrices, such as $\mathbf{v}\in\BR^{n}$ and $\mathbf{M}\in\BR^{m\times n}$. For a (column) vector $\mathbf{v}\in\BR^{n}$, $\mathbf{v}^{\mathrm{T}}$ represents its transpose, and $\|\mathbf{v}\|_{2}$ its Euclidean norm. $\mathbf{0}_{n\times m}$ is the zero matrix of dimension $n\times m$ and $\mathbf{I}_{n\times n}$ is the identity matrix of dimension $n\times n$.

\section{COOPERATIVE SWARM CONTROL PROBLEM}\label{Sec:Problem}

Consider a swarm of $m=1,2,\dots,M$ autonomous robots, each following the linear (double integrator) dynamics
\begin{equation}\label{Equ:DTSystem}
  \mathbf{s}^{(m)}_{k+1}=\mathbf{A}\mathbf{s}^{(m)}_{k}+\mathbf{B}\mathbf{a}^{(m)}_{k}
\end{equation}
in discrete time $k\in\mathbb{Z}_{0+}$, for a sampling time $\Delta t\in\BR_{+}$. The state vector $\mathbf{s}^{(m)}_{k}\in\mathbb{R}^{6}$ of agent $m$ in time step $k$ comprises of the 3-dimensional position and velocity vectors, using the notation
\begin{subequations}\label{Equ:StateVector}
\begin{equation}\label{Equ:StateVector1}
  \mathbf{s}^{(m)}_{k}=
  \begin{bmatrix}
  	\mathbf{p}^{(m)}_{k}\\
	\mathbf{v}^{(m)}_{k}
  \end{bmatrix}\ec
\end{equation}\vspace*{-0.3cm}

\noindent and\vspace*{0.1cm}
\begin{equation}\label{Equ:StateVector2}
  \mathbf{p}^{(m)}_{k}=
  \begin{bmatrix}
  	x_{k}^{(m)}\\
	y_{k}^{(m)}\\
	z_{k}^{(m)}
  \end{bmatrix}\ec\quad
  \mathbf{v}^{(m)}_{k}=
  \begin{bmatrix}
  	v_{\mathrm{x},k}^{(m)}\\
	v_{\mathrm{y},k}^{(m)}\\
	v_{\mathrm{z},k}^{(m)}
  \end{bmatrix}\ef
\end{equation}
\end{subequations}
The input vector $\mathbf{a}^{(m)}_{k}\in\mathbb{R}^{3}$ of agent $m$ in time step $k$ consists of the three spacial accelerations
\begin{equation}\label{Equ:InputVector}
  \mathbf{a}^{(m)}_{k}=
  \begin{bmatrix}
  	a_{\mathrm{x},k}^{(m)}\\
	a_{\mathrm{y},k}^{(m)}\\
	a_{\mathrm{z},k}^{(m)}
  \end{bmatrix}\ef
\end{equation}
Each agent $m$ starts from an initial position with an initial velocity, i.e., an initial state $\mathbf{s}^{(m)}_{0}\in\mathbb{R}^{6}$.

\begin{assumption}[\textbf{Dynamics}]\label{The:Dynamics}
	(a) The system matrix $\mathbf{A}\in\BR^{6\times 6}$ and input matrix $\mathbf{B}\in\BR^{6\times 3}$ are given by independent double integrator dynamics in each spacial direction:
	\begin{equation*}
		\mathbf{A}\eqdef
		\begin{bmatrix}
			\mathbf{I}_{3\times 3} & \Delta t\mathbf{I}_{3\times 3}\\
			\mathbf{0}_{3\times 3} & \mathbf{I}_{3\times 3}
		\end{bmatrix}\qquad\text{and}\qquad
		\mathbf{B}\eqdef
		\begin{bmatrix}
			\frac{\left(\Delta t\right)^2}{2}\mathbf{I}_{3\times 3}\\
			\Delta t\mathbf{I}_{3\times 3}\ef
		\end{bmatrix}
	\end{equation*}
	(b) The system model is known without uncertainty and there are no external disturbances to the system.
\end{assumption}

\begin{assumption}[\textbf{Dimensions}]\label{The:AgentDim}
  Each agent $m$ is fully contained inside a sphere with radius $\rho\in\BR_{+}$, centered in its current position. For collision avoidance, there may be no overlap between the individual spheres.
\end{assumption}

The acceleration vectors of all agents $m$ are constrained by
\begin{equation}\label{Equ:AccConstraints}
	a^{(m)}_{k}=\left\|\mathbf{a}^{(m)}_{k}\right\|_{2}\leq \bar{a}\qquad\fa k\in\BZ_{0+}\ec
\end{equation}
where $\bar{a}\in\BR_{+}$ denotes the uniform acceleration bound. The notation $a^{(m)}_{k}$ represents the magnitude of $\mathbf{a}^{(m)}_{k}$. Similarly, the velocity of agent $m$ is constrained by
\begin{equation}\label{Equ:VelConstraints}
	v^{(m)}_{k}=\left\|\mathbf{v}^{(m)}_{k}\right\|_{2}\leq \bar{v}\qquad\fa k\in\BZ_{0+}\ec
\end{equation}
for some $\bar{v}\in\BR_{+}$. A velocity bound, however, is not strictly necessary. Each agent of the swarm follows an individual objective, and must remain inside an individual state constraint set, including static obstacles:
\begin{equation}\label{Equ:StateConstraints}
	\mathbf{p}^{(m)}_{k}\in\mathbb{S}^{(m)}\subseteq\BR^{3}\ef
\end{equation}

\begin{assumption}[\textbf{Swarm}]\label{The:Swarm}
  (a) The dynamics \eqref{Equ:DTSystem}, the dynamic capabilities \eqref{Equ:AccConstraints},\eqref{Equ:VelConstraints}, and the agent size $\rho$ are identical for and known to all agents. (b) In every time step $k$, each agent $m$ can measure its own state and that of all other agents $j=1,2,\dots,M$, $j\neq m$ precisely. In particular, there is no measurement noise. (c) All agents share the same sampling time $\Delta t$ and operate on a synchronized clock. 
\end{assumption}

\begin{assumption}[\textbf{Communication}]\label{The:Communication}
  During the swarm operation, there is no communication between the agents. In particular, the agents are unable to share any information about their individual objective or position constraints \eqref{Equ:StateConstraints}, or to coordinate their motion plans, e.g., via Lagrange multipliers or a cooperative assignment of safe regions.
\end{assumption}

The goal of CMC is to establish a framework for the swarm that ensures persistent constraint satisfaction and collision avoidance between the robots. The framework consists of a set of cooperative rules that the agents of the swarm agree upon offline, i.e., before the actual operation.

\section{CONTINGENCY PLAN AND COLLISION AVOIDANCE CONSTRAINTS}\label{Sec:CoCoDefinition}

The CMC algorithm is based on the concept of a contingency trajectory, which is used, in a subsequent step, to set up the collision avoidance constraints. This Section provides their detailed derivation.

\subsection{Contingency Plan}\label{Sec:ContPlan}

Let $N\in\BN$ be the prediction horizon. In order to maintain recursive feasibility, a \emph{contingency trajectory} is defined for each agent. It is based on a \emph{contingency horizon}, which is defined as the smallest number of steps in which the agent can reach a stopping position: 
\begin{equation}\label{Equ:ContingencyHorizon}
	\tilde{N}_{k}^{(m)}\eqdef\left\lceil\frac{v_{k}^{(m)}}{\bar{a}\Delta t}\right\rceil\ef
\end{equation}
Here $\lceil*\rceil$ denotes the smallest integral number that is greater than $*$. Obviously, the contingency horizon is upper bounded by
  \begin{equation}\label{Equ:MaxContingencyHorizon}
	\tilde{N}_{k}^{(m)}\leq\tilde{N}\eqdef\left\lceil\frac{\bar{v}}{\bar{a}\Delta t}\right\rceil\ef
\end{equation}
In the remainder of the paper, it is assumed that the maximum contingency horizon is strictly smaller than the prediction horizon, i.e., $\tilde{N}\leq N-1$. This assumption could be relaxed without major implications, other than notational inconveniences.

\begin{remark}[\textbf{Contingency horizon}]
  (a) The contingency horizon is time-varying. (b) In time step $k$, it is determined only by the velocity $\mathbf{v}_{k}^{(m)}$ of agent $m$ and static information.
\end{remark}

The \emph{contingency trajectory} is defined by a constant deceleration pattern over the contingency horizon; see Figure \ref{Fig:ContingencyPlan} for an illustration. The deceleration is chosen such that the direction of motion is preserved, i.e.,
\begin{subequations}
\begin{equation}\label{Equ:ContingencyAcceleration1}
	\mathbf{\tilde{a}}_{k}^{(m)}\eqdef \tilde{a}_{k}^{(m)}\frac{\mathbf{v}_{k}^{(m)}}{v_{k}^{(m)}}\ef
\end{equation}
Here $\tilde{a}_{k}^{(m)}\in\BR$ is defined as the signed (scalar) acceleration in the direction of the velocity $\mathbf{v}_{k}^{(m)}$. In order to reach a stopping position, this signed acceleration must by chosen as
\begin{equation}\label{Equ:ContingencyAcceleration2}
	\tilde{a}_{k}^{(m)}=-\frac{v_{k}^{(m)}}{\tilde{N}_{k}^{(m)}\Delta t}
\end{equation}
\end{subequations}
over the subsequent $\tilde{N}_{k}^{(m)}$ time steps, and equal to $0$ thereafter. Combining \eqref{Equ:ContingencyAcceleration1} and \eqref{Equ:ContingencyAcceleration2} yields
\begin{equation}\label{Equ:ContingencyAcceleration}
	\mathbf{\tilde{a}}_{i|k}^{(m)}=
	\left\{\begin{array}{ll}
	-\frac{\mathbf{v}_{k}^{(m)}}{\tilde{N}_{k}^{(m)}\Delta t}& \text{for}\;0\leq i\leq \tilde{N}_{k}^{(m)}-1\\
	\mathbf{0}_{3} & \text{for}\;i\geq\tilde{N}_{k}^{(m)}
	\end{array}\right.\ef
\end{equation}
Here, and in the following, the double index `$i|k$' is used to denote the $i$-step prediction or planning step made at time $k$.

The (negative) acceleration in \eqref{Equ:ContingencyAcceleration} is always feasible, in the sense that its magnitude is less than $\bar{a}$. This follows immediately from the definition of the contingency horizon \eqref{Equ:ContingencyHorizon}.

In summary, the \emph{contingency trajectory} is such that the agent follows a straight path along its initial direction of motion, with a linearly decreasing velocity to a standstill.


\begin{figure}[H]
    \vspace*{-0.2cm}
    \includegraphics[scale=1]{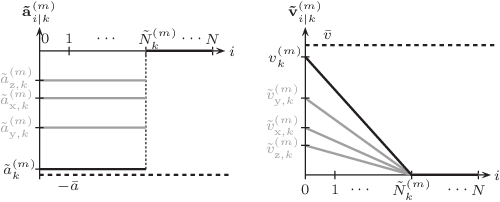} 
	\caption{Illustration of the acceleration profile (left) and the velocity profile (right) of the contingency trajectory for agent $m$ in step $k$.\label{Fig:ContingencyPlan}}
\end{figure}

\begin{remark}[\textbf{Contingency plan}]\label{Rem:ContingencyPlan}
  (a) The contingency plan for each agent $m$ consists of a contingency horizon and a contingency trajectory. (b) Both are completely determined by the measured state $\mathbf{s}_{k}^{(m)}$ and do not contain any further decision-making element. (c) Hence the contingency plan can be computed by each agent $m$ for itself, and all other agents $j=1,2,\dots,M$, $j\neq m$, based on Assumptions \ref{The:Dynamics},\ref{The:Swarm}. 
\end{remark}

\subsection{Collision Avoidance Constraints}\label{Sec:CollisionAvoidance}

Collision avoidance constraints are introduced pairwise, between every two agents $m$ and $j$. For the remainder of this section, assume the perspective of a fixed agent $m$, and consider the setup of the collision avoidance constraints with respect to another agent $j$.

The setup is based on the contingency trajectories of agents $m$ and $j$. According to Remark \ref{Rem:ContingencyPlan}, both of them can be constructed independently by (each) agent $m$ for (every other) agent $j$. Consider a fixed time step $k$. The (predicted) positions of the contingency trajectory by agent $m$ for itself are denoted
\begin{equation}
	\mathbf{\tilde{p}}_{i|k}^{(m)}\qquad\fa i=0,1,\dots,\tilde{N}\ec
\end{equation}
where $\mathbf{\tilde{p}}_{0|k}^{(m)}=\mathbf{p}_{k}^{(m)}$. Equivalently, the contingency positions for agent $j$ are denoted
\begin{equation}
	\mathbf{\tilde{p}}_{i|k}^{(j)}\qquad\fa i=0,1,\dots,\tilde{N}\ec
\end{equation}
where $\mathbf{\tilde{p}}_{0|k}^{(j)}=\mathbf{p}_{k}^{(j)}$.


The setup of collision avoidance constraints involves a few simple geometric computations. Figure \ref{Fig:CollAvoidConstraint} provides an illustration. The goal for agent $m$ is not to pass over the perpendicular bisecting plane of the direct connection between the contingency positions of agents $m$ and $j$, in each prediction step $i=1,2,\dots,\tilde{N}$. To this end, the perpendicular bisecting plane has to be tightened by the size of the robot, $\rho$.

The first step is to compute the (normalized) vector pointing from the contingency position of agent $m$ to that of agent $j$:
\begin{subequations}
\begin{equation}
	\mathbf{g}_{i|k}^{(j,m)}\eqdef\frac{\mathbf{\tilde{p}}_{i|k}^{(j)}-\mathbf{\tilde{p}}_{i|k}^{(m)}}{d_{i|k}^{(j,m)}}\ec
\end{equation}
where
\begin{equation}\label{Equ:AgentDistances}
	d_{i|k}^{(j,m)}\eqdef\left\|\mathbf{\tilde{p}}_{i|k}^{(j)}-\mathbf{\tilde{p}}_{i|k}^{(m)}\right\|_2
\end{equation}
\end{subequations}
is the distance between the two agents. The second step is to compute the tightening,
\begin{equation}
	h_{i|k}^{(j,m)}\eqdef\mathbf{g}_{i|k}^{(j,m)\mathrm{T}}\mathbf{\tilde{p}}_{i|k}^{(m)}+\left(\frac{1}{2}d_{i|k}^{(j,m)}-\rho\right)\ef
\end{equation}
So the final collision avoidance constraint reads as
\begin{equation}\label{Equ:CollAvoidConstraints}
	\mathbf{g}_{i|k}^{(j,m)\mathrm{T}}
	\begin{bmatrix}
	  x\\
	  y\\
	  z
	\end{bmatrix}\leq
	h_{i|k}^{(j,m)}\ec
\end{equation}
for all $i=1,2,\dots,\tilde{N}$.

\begin{figure}[H]
	\centering
    \includegraphics[scale=1]{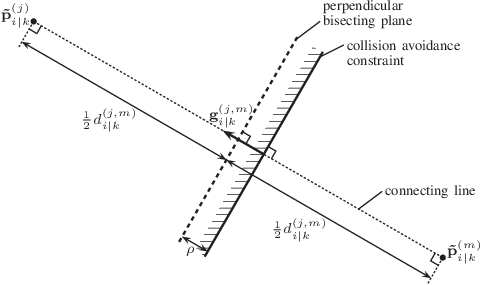} 
	\caption{Construction of the collision avoidance constraint for agent $m$, with respect to agent $j$, in the prediction step $i$.\label{Fig:CollAvoidConstraint}}
\end{figure}

\begin{remark}[\textbf{Collision avoidance constraints}]
  In time step $k$, there is exactly one linear collision avoidance constraint for agent $m$ with respect to each other agent $j$ and in each prediction step $i=1,2,\dots,\tilde{N}$.
\end{remark}

\section{COOPERATIVE MODEL-BASED CONTROL (CMC)}\label{Sec:CMCSetup}

The proposed CMC algorithm involves an individual Finite-time Optimal Control Problem (FTOCP), to be solved by each agent $m$ in each time step $k$. This setup of the FTOCP is described in this section.

\subsection{Nominal Predictions}

The decision vector for agent $m$ in step $k$ includes the (nominal) predicted control inputs
\begin{equation}\label{Equ:PredictedInputs}
	\mathbf{a}_{0|k}^{(m)}\ec\quad\mathbf{a}_{1|k}^{(m)}\ec\quad\dots\ec\quad
	\mathbf{a}_{N-1|k}^{(m)}\ec
\end{equation}
each of which must satisfy the input constraint \eqref{Equ:AccConstraints}. The corresponding (nominal) predicted state sequence is denoted
\begin{equation}\label{Equ:PredictedStates}
	\mathbf{s}_{0|k}^{(m)}\ec\quad\mathbf{s}_{1|k}^{(m)}\ec\quad\dots\ec\quad
	\mathbf{s}_{N|k}^{(m)}\ef
\end{equation}
It starts at the measured state $\mathbf{s}_{0|k}^{(m)}=\mathbf{s}_{k}^{(m)}$. The predicted states can be expressed in terms of the measured state and the predicted control inputs:
\begin{equation}\label{Equ:NominalStatePredictions}
	\mathbf{s}_{i|k}^{(m)}=\mathbf{A}^i\mathbf{s}_{k}^{(m)}+\sum_{l=0}^{i-1}\mathbf{A}^{i-l-1}\mathbf{B}\mathbf{a}_{l|k}^{(m)}\ec
\end{equation}
using the linear dynamics \eqref{Equ:DTSystem}. 

\subsection{Contingency Predictions}

According to the receding-horizon scheme of CMC, only the first of the predicted control inputs, $\mathbf{a}_{0|k}^{(m)}$, is applied to the system. For recursive feasibility of the scheme, it must be ensured that a feasible contingency trajectory is also available in the next time step, i.e., starting from the state
\begin{equation}\label{Equ:OneStepState}
	\mathbf{s}_{1|k}^{(m)}\eqdef
	\begin{bmatrix}
		\mathbf{p}_{1|k}^{(m)}\\
		\mathbf{v}_{1|k}^{(m)}
	\end{bmatrix}\ef
\end{equation}

According to Section \ref{Sec:ContPlan}, this contingency trajectory is fully determined by the predicted position $\mathbf{p}_{1|k}^{(m)}$ and the predicted velocity $\mathbf{v}_{1|k}^{(m)}$. It is defined by a constant deceleration, to a standstill, along a straight path along the direction of $\mathbf{v}_{1|k}^{(m)}$. If the new contingency horizon $\tilde{N}_{k+1}^{(m)}$ were known, then all contingency inputs could simply be calculated as
\begin{equation}\label{Equ:ContingencyInputs}
	\mathbf{\tilde{a}}_{i|k}^{(m)}=
	\left\{\begin{array}{ll}
	-\frac{\mathbf{v}_{1|k}^{(m)}}{\tilde{N}_{k+1}^{(m)}\Delta t}& \text{for}\;1\leq i\leq \tilde{N}_{k+1}^{(m)}\\
	\qquad\mathbf{0}_{3} & \text{for}\;\tilde{N}_{k+1}^{(m)}<i\leq \tilde{N}
	\end{array}\right.\ef
\end{equation}


Here $\mathbf{\tilde{a}}_{i|k}^{(m)}$ denotes the contingency control input for the prediction step $i=1,2,\dots,\tilde{N}$. The tilde is generally used to denote variables related to the contingency trajectory. However, the new contingency horizon $\tilde{N}_{k+1}^{(m)}$ is unknown in step $k$, as it may differ from $\tilde{N}_{k}^{(m)}$.

\begin{proposition}[Change of the contingency horizon]
  The contingency horizon in step $k+1$ satisfies
  \begin{subequations}\label{Equ:ChangeContHorizon}
  \begin{equation}\label{Equ:ChangeContHorizon1}
	\tilde{N}_{k+1}^{(m)}\in\left\{\tilde{N}_{k}^{(m)}\hspace*{-0.1cm}-\hspace*{-0.05cm}1\,,\,\tilde{N}_{k}^{(m)}\,,\,\tilde{N}_{k}^{(m)}\hspace*{-0.1cm}+\hspace*{-0.05cm}1\right\}
  \end{equation}
  and
  \begin{equation}\label{Equ:ChangeContHorizon2}
	0\leq\tilde{N}_{k+1}^{(m)}\leq\tilde{N}\ef
  \end{equation}
  \end{subequations}
\end{proposition}

\begin{proof}
	Following \eqref{Equ:ContingencyHorizon}, the new contingency horizon
	\begin{equation*}
		\tilde{N}_{k+1}^{(m)}\eqdef\left\lceil\frac{v_{k+1}^{(m)}}{\bar{a}\Delta t}\right\rceil\ef
	\end{equation*}
	depends on the magnitude $v_{k+1}^{(m)}$ of the velocity $\mathbf{v}_{k+1}^{(m)}$. By Assumption \ref{The:Dynamics},
	\begin{equation}\label{Equ:ChangeContHorizonProof1}
		\mathbf{v}_{k+1}^{(m)}=\mathbf{v}_{1|k}^{(m)}=\mathbf{v}_{k}^{(m)}+\Delta t\,\mathbf{a}_{0|k}^{(m)}\ec
	\end{equation}
	where $\mathbf{v}_{k}^{(m)}$ is the (fixed) velocity in step $k$ and $\mathbf{a}_{0|k}^{(m)}$ is the (variable) first nominal predicted input. Condition \eqref{Equ:ChangeContHorizon2} follows from the velocity constraint, due to which
	\begin{equation*}
		0\leq v_{k+1}^{(m)}\leq\bar{v}\ef
	\end{equation*}
	Condition \eqref{Equ:ChangeContHorizon1} follows from the acceleration constraint
	\begin{equation*}
		\left\|\mathbf{a}_{0|k}^{(m)}\right\|_{2}\leq\bar{a}
	\end{equation*}
	in combination with \eqref{Equ:ChangeContHorizonProof1}. The triangle inequality yields
	\begin{equation}\label{Equ:ChangeContHorizonProof2}
		 \left\|\mathbf{v}_{k}^{(m)}\right\|_{2}-\bar{a}\Delta t\leq\left\|\mathbf{v}_{k+1}^{(m)}\right\|_{2}\leq \left\|\mathbf{v}_{k}^{(m)}\right\|_{2}+\bar{a}\Delta t\ec
	\end{equation}
	proving the claim. The geometry situation is illustrated in Figure \ref{Fig:NewContingencyTraj}.
	\end{proof}
	
	\begin{figure}[H]
	\centering
    \includegraphics[scale=1]{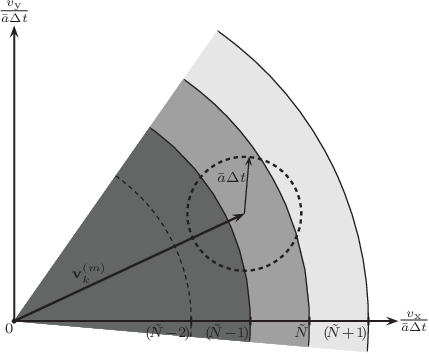} 
	\caption{Composition of the initial velocity $\mathbf{v}_{1|k}^{(m)}$ for the new contingency trajectory. Since it is the sum of $\mathbf{v}_{k}^{(m)}$ (fixed) and $\mathbf{a}_{0|k}^{(m)}\Delta t$ (decision variable), it must lie inside the bold dotted circle.\label{Fig:NewContingencyTraj}}
\end{figure}

Based on the contingency control inputs \eqref{Equ:ChangeContHorizon}, the states of the (predicted) contingency trajectory are given by
\begin{subequations}\label{Equ:ContingencyStates}
\begin{equation}\label{Equ:ContingencyStates1}
	\mathbf{\tilde{s}}_{i|k}^{(m)}=\mathbf{A}^i\mathbf{s}_{k}^{(m)}+\mathbf{A}^{i-1}\mathbf{B}\mathbf{a}_{0|k}^{(m)}+\sum_{l=1}^{i-1}\mathbf{A}^{i-l-1}\mathbf{B}\mathbf{\tilde{a}}_{l|k}^{(m)}
\end{equation}
for $i=1,2,\dots,\tilde{N}_{k+1}^{(m)}+1$. Beyond this, for all predictions steps $i\geq\tilde{N}_{k+1}^{(m)}+1$, the contingency state remains constant,
\begin{equation}\label{Equ:ContingencyStates2}
	\mathbf{\tilde{s}}_{i|k}^{(m)}=\mathbf{\tilde{s}}_{\tilde{N}_{k+1}^{(m)}+1|k}^{(m)}\ec
\end{equation}
\end{subequations}
because the corresponding contingency velocities zero. In \eqref{Equ:ContingencyStates1}, the contingency inputs $\mathbf{\tilde{a}}_{l|k}^{(m)}$ can be substituted by \eqref{Equ:ContingencyInputs}. Hence the contingency trajectory in the FTOCP is fully parameterized, except for the unknown contingency horizon $\tilde{N}_{k+1}^{(m)}$.

\subsection{Finite-time Optimal Control Problem (FTOCP)}\label{Sec:SetupFTOCP}

Consider again a fixed agent $m$. The cost function $J^{(m)}(\cdot)$ of the FTOCP can be any generic function of the predicted nominal inputs and states,
\begin{subequations}\label{Equ:CostFunction}
\begin{equation}\label{Equ:NomCostFunction}
	J^{(m)}\left(\mathbf{a}_{0|k}^{(m)},\dots,\mathbf{a}_{N-1|k}^{(m)},\mathbf{s}_{1|k}^{(m)},\dots,\mathbf{s}_{N|k}^{(m)}\right)\ef
\end{equation}
Optionally, it could be augmented by an (identical or different) cost function $\tilde{J}^{(m)}(\cdot)$ of the contingency inputs and states,
\begin{equation}\label{Equ:CntCostFunction}
	\tilde{J}^{(m)}\left(\mathbf{\tilde{a}}_{1|k}^{(m)},\dots,\mathbf{\tilde{a}}_{\tilde{N}_{k+1}^{(m)}|k}^{(m)},\mathbf{\tilde{s}}_{2|k}^{(m)},\dots,\mathbf{\tilde{s}}_{\tilde{N}_{k+1}^{(m)}+1|k}^{(m)}\right)\ef
\end{equation}
\end{subequations}
The cost function(s) may be a sum of stage cost terms. It (They) may be related, for instance, to reaching a target state or tracking a predefined reference trajectory.

Combining the previous derivations \eqref{Equ:CostFunction}, \eqref{Equ:NominalStatePredictions}, \eqref{Equ:ContingencyAcceleration}, \eqref{Equ:ContingencyStates}, and \eqref{Equ:CollAvoidConstraints}, the FTOCP reads as follows:
\begin{subequations}\label{Equ:FTOCP}
\begin{align}
	&\min\; J^{(m)}\left(\mathbf{a}_{0|k}^{(m)},\dots,\mathbf{a}_{N-1|k}^{(m)},\mathbf{s}_{1|k}^{(m)},\dots,\mathbf{s}_{N|k}^{(m)}\right)\\
	&\mathrm{subject}\;\mathrm{to}\nonumber\\
	&\mathbf{s}_{i|k}^{(m)}=\mathbf{A}^i\mathbf{s}_{k}^{(m)}+\sum_{l=0}^{i-1}\mathbf{A}^{i-l-1}\mathbf{B}\mathbf{a}_{l|k}^{(m)}\quad\fa i\hspace*{-0.05cm}=\hspace*{-0.05cm}1,...,N\\
	&\left\|\mathbf{a}_{i|k}^{(m)}\right\|_{2}\leq \bar{a}\qquad\fa i\hspace*{-0.05cm}=\hspace*{-0.05cm}0,...,N-1\\
	&\mathbf{s}_{i|k}^{(m)}\in\mathbb{S}^{(m)}\qquad\fa i\hspace*{-0.05cm}=\hspace*{-0.05cm}1,...,N\\
	&\mathbf{\tilde{a}}_{i|k}^{(m)}=\textstyle{\frac{-1}{\hat{N}_{k+1}^{(m)}\Delta t}}\mathbf{v}_{k}^{(m)}+\textstyle{\frac{-1}{\hat{N}_{k+1}^{(m)}}}\mathbf{a}_{0|k}^{(m)}\quad\fa i\hspace*{-0.05cm}=\hspace*{-0.05cm}1,...,\hat{N}_{k+1}^{(m)}\hspace*{-0.1cm}\\
	& \mathbf{\tilde{s}}_{i|k}^{(m)}=\mathbf{A}^i\mathbf{s}_{k}^{(m)}+\mathbf{A}^{i-1}\mathbf{B}\mathbf{a}_{0|k}^{(m)}+\sum_{l=1}^{i-1}\mathbf{A}^{i-l-1}\mathbf{B}\mathbf{\tilde{a}}_{l|k}^{(m)}\nonumber\\
	&\hspace*{4.6cm}\fa i\hspace*{-0.05cm}=\hspace*{-0.05cm}1,...,\hat{N}_{k+1}^{(m)}+1\\
	&\mathbf{\tilde{s}}_{i|k}^{(m)}\in\mathbb{S}^{(m)}\qquad\fa i\hspace*{-0.05cm}=\hspace*{-0.05cm}2,...,\hat{N}_{k+1}^{(m)}+1\\
	&\left[\mathbf{g}_{i|k}^{(j,m)\mathrm{T}}\;0\;0\;0\right]\mathbf{\tilde{s}}_{i|k}^{(m)}\leq h_{i|k}^{(j,m)}\nonumber\\
	&\hspace*{0.7cm}\fa i\hspace*{-0.05cm}=\hspace*{-0.05cm}1,...,\hat{N}_{k+1}^{(m)}+1\ec\quad\fa j\hspace*{-0.05cm}=\hspace*{-0.05cm}1,...,M\ec\;j\hspace*{-0.05cm}\neq m
\end{align}
\end{subequations}
Conceptually---yet possibly not computationally---the only decision variables of \eqref{Equ:FTOCP} are the nominal control inputs. The initial condition $\mathbf{s}_{k}^{(m)}$ in (\ref{Equ:FTOCP}b) and (\ref{Equ:FTOCP}f) is given and fixed. The predicted states follow from the linear equality conditions (\ref{Equ:FTOCP}b). The contingency inputs (\ref{Equ:FTOCP}e), and hence the contingency states (\ref{Equ:FTOCP}f), can be computed solely based on the first nominal control input.

The collision avoidance constraints (\ref{Equ:FTOCP}g) need to be enforced only for the contingency states. The contingency inputs (\ref{Equ:FTOCP}e) automatically satisfy the input constraints, if $\hat{N}_{k+1}^{(m)}$ is a \emph{valid} contingency horizon. That is, if it turns out to satisfy condition \eqref{Equ:ContingencyHorizon}:
\begin{equation}\label{Equ:ContingencyHorizonCand}
	\hat{N}_{k+1}^{(m)}\overset{?}{=}\left\lceil\frac{v_{1|k}^{(m)\star}}{\bar{a}\Delta t}\right\rceil\ec
\end{equation}
where $v_{1|k}^{(m)\star}$ is the magnitude of the one-step predicted velocity $\mathbf{v}_{1|k}^{(m)\star}$ (which is a part of $\mathbf{s}_{1|k}^{(m)\star}$). In other words, $\hat{N}_{k+1}^{(m)}$ is used as a candidate for the contingency horizon for solving the FTOCP in time step $k$, since $\tilde{N}_{k+1}^{(m)}$ is unknown. The candidate horizon needs to be confirmed afterwards, based on the optimal solution, by checking \eqref{Equ:ContingencyHorizonCand}. The details on how this is handled are discussed in Section \ref{Sec:CMCAlgorithm}.

The following result about the computational complexity of the FTOCP \eqref{Equ:FTOCP} is important. The proof, however, is obvious and therefore omitted. 

\begin{proposition}[Convexity of the FTOCP]
	Suppose\vspace*{-0.05cm}
	\begin{itemize}
		\item[(1)] the cost function \eqref{Equ:CostFunction} is convex, and
		\item[(2)] the constraint set $\mathbb{S}^{(m)}$ is convex.
	\end{itemize}\vspace*{-0.05cm}
	Then the FTOCP \eqref{Equ:FTOCP} is a convex optimization problem for each agent $m\in\{1,2,\dots,M\}$ and every choice of $\hat{N}_{k+1}^{(m)}$.
\end{proposition}

Denote the optimal solution of \eqref{Equ:FTOCP} with a `$\star$'. In particular, 
	\begin{equation}\label{Equ:OptimalNomPredictions}
		\mathbf{a}_{0|k}^{(m)\star},\;\dots,\;
		\mathbf{a}_{N-1|k}^{(m)\star}\qquad\text{and}\qquad
		\mathbf{s}_{0|k}^{(m)\star},\;\dots,\;
	\mathbf{s}_{N|k}^{(m)\star}
	\end{equation}
	are the optimal nominal control inputs and the corresponding optimal nominal states. And
	\begin{subequations}\label{Equ:OptimalCtgPredictions}
	\begin{equation}\label{Equ:OptimalCtgPredictions1}
		\mathbf{\tilde{a}}_{1|k}^{(m)\star},\;\dots,\;
		\mathbf{\tilde{a}}_{\hat{N}_{k+1}^{(m)}|k}^{(m)\star}\quad\text{and}\quad
		\mathbf{\tilde{s}}_{1|k}^{(m)\star},\;\dots,\;
	\mathbf{\tilde{s}}_{\hat{N}_{k+1}^{(m)}+1|k}^{(m)\star}
	\end{equation}
	are the optimal contingency control inputs and the corresponding optimal contingency states. After the solution to \eqref{Equ:FTOCP}, the optimal contingency inputs and states \eqref{Equ:OptimalCtgPredictions1} can be naturally augmented by
	\begin{align}\label{Equ:OptimalCtgPredictions2}
		&\mathbf{\tilde{a}}_{\hat{N}_{k+1}^{(m)}+1|k}^{(m)\star}=\dots=\mathbf{\tilde{a}}_{N-1|k}^{(m)\star}=\mathbf{0}_{3}\ec\\
		&\mathbf{\tilde{s}}_{\hat{N}_{k+1}^{(m)}+1|k}^{(m)\star}=\mathbf{\tilde{s}}_{\hat{N}_{k+1}^{(m)}+2|k}^{(m)\star}=\dots=\mathbf{\tilde{s}}_{N|k}^{(m)\star}\ec
	\end{align}
	\end{subequations}
	for notational convenience.

\subsection{CMC Algorithm}\label{Sec:CMCAlgorithm}

It remains to specify how a valid contingency horizon $\hat{N}_{k+1}^{(m)}$ is selected. To this end, the overall CMC algorithm is composed as follows.\vspace*{0.2cm}

\noindent\fbox{\begin{minipage}{8.4cm}
Steps to be performed by each agent $m$ in every step $k$:
\begin{itemize}
	\item[(1)] Generate the contingency predictions for agent $m$ itself, and all other agents $j=1,2,\dots,M$, $j\neq m$, as described Section \ref{Sec:ContPlan}.
	\item[(2)] Compute the collision avoidance constraints with respect to all other agents $j=1,2,\dots,M$, $j\neq m$, and for all prediction steps $i=1,2,\dots,\max\{N,\tilde{N}\}$, as described in Section \ref{Sec:CollisionAvoidance}.
	\item[(3)] Set up the FTOCP for agent $m$, as described in Section \ref{Sec:SetupFTOCP}. Proceed through the following three steps:
	\begin{itemize}
		\item[(a)] Solve the FTOCP \eqref{Equ:FTOCP} for $\hat{N}_{k+1}^{(m)}\leftarrow\tilde{N}_{k}^{(m)}+1$ (unless $\tilde{N}_{k}^{(m)}=\tilde{N}$). If the solutions satisfies \eqref{Equ:ContingencyHorizonCand},\vspace*{0.1cm}\\
	\hspace*{2cm}$\displaystyle\hat{N}_{k+1}^{(m)}=\left\lceil\frac{v_{1|k}^{(m)\star}}{\bar{a}\Delta t}\right\rceil\ec$\vspace*{0.1cm}\\
		then stop and apply $\mathbf{a}_{0|k}^{(m)\star}$ to the system.
		\item[(b)] Else solve the FTOCP \eqref{Equ:FTOCP} for $\hat{N}_{k+1}^{(m)}\leftarrow\tilde{N}_{k}^{(m)}$, adding the constraints\vspace*{0.3cm}\\
			\hspace*{0.4cm}$\displaystyle\left\|
			\begin{bmatrix}
				\mathbf{0}_{3\times 3} & \mathbf{0}_{3\times 3}\\
				\mathbf{0}_{3\times 3} & \mathbf{I}_{3\times 3}
			\end{bmatrix}
			\mathbf{s}_{1|k}^{(m)}\right\|_{2}\leq\bar{a}\tilde{N}_{k}^{(m)}\Delta t\ef$	\hfill (\ref{Equ:FTOCP}i)\vspace*{0.3cm}\\
		If the solution satisfies \eqref{Equ:ContingencyHorizonCand}, then stop and apply $\mathbf{a}_{0|k}^{(m)\star}$ to the system.
		\item[(c)] Else solve the FTOCP \eqref{Equ:FTOCP} for $\hat{N}_{k+1}^{(m)}\leftarrow\tilde{N}_{k}^{(m)}-1$, adding the constraints\vspace*{0.3cm}\\
			\hspace*{0.1cm}$\displaystyle\left\|
			\begin{bmatrix}
				\mathbf{0}_{3\times 3} & \mathbf{0}_{3\times 3}\\
				\mathbf{0}_{3\times 3} & \mathbf{I}_{3\times 3}
			\end{bmatrix}
			\mathbf{s}_{1|k}^{(m)}\right\|_{2}\leq\bar{a}\hspace*{-0.05cm}\left(\hspace*{-0.05cm}\tilde{N}_{k}^{(m)}\hspace*{-0.1cm}-\hspace*{-0.1cm}1\hspace*{-0.05cm}\right)\hspace*{-0.05cm}\Delta t.$	\hfill (\ref{Equ:FTOCP}j)\vspace*{0.3cm}\\
		Apply $\mathbf{a}_{0|k}^{(m)\star}$ to the system.
	\end{itemize}
\end{itemize}
\end{minipage}}\vspace*{0.2cm}

\begin{remark}[\textbf{Limited sensor range}]
  (a) It is possible to omit collision avoidance for agent $j$ in the FTOCP of agent $m$, and vice versa, if their distance exceeds a certain threshold. A suitable threshold is (circa) twice the minimum braking distance of an agent from its maximum speed. (b) This modification can be used to reduce the number of constraints in the FTOCP and for a reasonable selection of a physical sensor range.
\end{remark}

\section{SYSTEM-THEORETICAL ANALYSIS}\label{Sec:TheoreticalAnalysis}

The goal of this section is to establish the recursive feasibility of the CMC algorithm, and the guarantee of no collisions between the agents and persistent constraint satisfaction in closed-loop operation.

\begin{lemma}[\textbf{Collision avoidance}]\label{The:CollisionAvoidance}
	Suppose the FTOCPs \eqref{Equ:FTOCP} for all agents $m=1,2,\dots,M$ are feasible in step $k$. Then (a) the optimal first control inputs $\mathbf{a}_{0|k}^{(m)\star}$ cause no collisions between the agents in step $k+1$, and (b) the contingency control inputs and states \eqref{Equ:OptimalCtgPredictions} are collision-free between all agents in all prediction steps $i=1,2,\dots,N$.
\end{lemma}

\begin{wrapfigure}{r}{0.6\columnwidth}
	\vspace*{0.1cm}
    \includegraphics[scale=1]{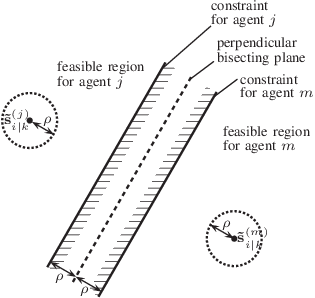} 
    \caption{Illustration of mutual collision avoidance.\label{Fig:ProofNoCollision}}
\end{wrapfigure}

\begin{proof}
Pick any two fixed agents, $j$ and $m$ ($j\neq m$), as shown in Figure \ref{Fig:ProofNoCollision}. By construction, the collision avoidance constraints for agents $j$ and $m$ in each prediction step $i=1,2,\dots,N$ are two parallel, linear constraints, respectively. Hence the feasible regions for the contingency states $\mathbf{\tilde{s}}_{i|k}^{(m)}$ and $\mathbf{\tilde{s}}_{i|k}^{(j)}$ are separated by two parallel (hyper-)planes. Their distance is $2\rho$, since each (hyper-)plane has the distance $\rho$ from the perpendicular bisecting (hyper-)plane. Given that each agent has a spherical shape with radius $\rho$, a collision is clearly impossible if $j$ and $m$ remain in their respective feasible regions, thus proving (b).

Due to Assumption \ref{The:Dynamics}(b) and based on (\ref{Equ:FTOCP}b,f), the states of agent $j$ and $m$ in step $k+1$ satisfy
\begin{equation}
 	\mathbf{s}_{k+1}^{(j)}=\mathbf{s}_{1|k}^{(j)\star}=\mathbf{\tilde{s}}_{1|k}^{(j)\star}\quad\text{and}\quad
	\mathbf{s}_{k+1}^{(m)}=\mathbf{s}_{1|k}^{(m)\star}=\mathbf{\tilde{s}}_{1|k}^{(m)\star}\ec
\end{equation}
i.e., they are equal to the first predicted nominal state and the first contingency state, respectively, thus proving (a).
\end{proof}

\begin{theorem}[\textbf{Recursive feasibility}]\label{The:RecFeasibility}
	Suppose all agents $m=1,2,\dots,M$ operate under the CMC regime. If the FTOCPs \eqref{Equ:FTOCP} of all agents are feasible in step $k=0$, then they will remain feasible for all agents in all future time steps $k=1,2,3,\dots$.
\end{theorem}

\begin{proof}
	By induction, a proof is required only for the feasibility of the FTOCPs in step $k=1$. Moreover, the proof may focus on the feasibility of the FTOCP for a arbitrary agent $m\in\{1,2,\dots,M\}$. Finally, it is enough to prove the feasibility of the FTOCP for the minimal possible contingency horizon, i.e.,
	\begin{subequations}
	\begin{equation}
		\hat{N}_{2}^{(m)}=\max\left\{0\;,\;\tilde{N}_{1}^{(m)}\hspace*{-0.1cm}-\hspace*{-0.05cm}1\right\}\ef
	\end{equation}
	That is because in time step $k=1$ it suffices if only one of the FTOCPs in steps (3.a),(3.b),(3.c) of the CMC algorithm is feasible. Thus, in the following, consider the FTOCP for agent $m$ in time step $k=1$ under the assumption that
	\begin{equation}\label{Equ:ProofCtgHorizon}
		\hat{N}_{2}^{(m)}=\tilde{N}_{1}^{(m)}\hspace*{-0.1cm}-\hspace*{-0.05cm}1\ef
	\end{equation}
	\end{subequations}
	The special case where $\tilde{N}_{1}^{(m)}=0$ is discussed separately at the end of the proof.
	
	First, note that the contingency trajectory is always a feasible choice for the predicted nominal trajectory. This holds because (i) the contingency inputs (\ref{Equ:FTOCP}e) satisfy the input constraints and (ii) the contingency states satisfy the state constraints (\ref{Equ:FTOCP}g). Hence it suffices to prove the existence of a feasible contingency trajectory in time step $k=1$.

	Recall that the optimal contingency input sequence in time step $k=0$ is denoted
	\begin{equation}\label{Equ:OptimalCtgPredictionsRep}
		\mathbf{\tilde{a}}_{1|0}^{(m)\star},\;\mathbf{\tilde{a}}_{2|0}^{(m)\star},\;\dots,\;
		\mathbf{\tilde{a}}_{\tilde{N}_{1}^{(m)}|0}^{(m)\star}\ef
	\end{equation}
	The claim is the following time shift of this sequence,
	\begin{subequations}\label{Equ:ShiftedInputSeq}
		\begin{align}
			\mathbf{a}_{0|1}^{(m)}&=\mathbf{\tilde{a}}_{1|0}^{(m)\star}\qquad\text{and}\\
			\mathbf{\tilde{a}}_{i|1}^{(m)}&=\mathbf{\tilde{a}}_{i+1|0}^{(m)\star}\qquad\fa 1\leq i\leq \tilde{N}_{1}^{(m)}-1\ec
		\end{align}
	\end{subequations}
	constitutes a feasible candidate for the first input and the contingency trajectory in time step $k=1$. It is obvious that (\ref{Equ:ShiftedInputSeq}a) satisfies the input constraint (\ref{Equ:FTOCP}c). Moreover, the contingency input candidates (\ref{Equ:ShiftedInputSeq}b) satisfy condition (\ref{Equ:FTOCP}e), due to \eqref{Equ:ProofCtgHorizon}. The corresponding candidate contingency state sequence is given by
	\begin{subequations}\label{Equ:ShiftedStateSeq}
		\begin{align}
		\mathbf{\tilde{s}}_{i|1}^{(m)}&=\mathbf{\tilde{s}}_{i+1|0}^{(m)\star}\qquad\;\,\,\fa 1\leq i\leq \tilde{N}_{1}^{(m)}\ec\\
		\mathbf{\tilde{s}}_{i|1}^{(m)}&=\mathbf{\tilde{s}}_{\tilde{N}_{1}^{(m)}|1}^{(m)}\qquad\fa \tilde{N}_{1}^{(m)}\leq i \leq \tilde{N}\ef
		\end{align}
	\end{subequations}
	Hence the state constraints (\ref{Equ:FTOCP}g) are also satisfied. The critical aspect is to prove the satisfaction of the collision avoidance constraints (\ref{Equ:FTOCP}h) in time step $k=1$. Note that these collision avoidance constraints are different from those in previous the time step $k=0$.
	
	To prove satisfaction of the collision avoidance constraints, fix another agent $j$ ($j\neq m$). Consider the two FTOCPs \eqref{Equ:FTOCP} for agents $j$ and $m$ in time step $k=0$. The optimal inputs $\mathbf{a}_{0|0}^{(j)\star}$ and $\mathbf{a}_{0|0}^{(m)\star}$ are generally different from the respective contingency accelerations $\mathbf{\tilde{a}}_{0|0}^{(j)}$ and $\mathbf{\tilde{a}}_{0|0}^{(m)}$ in \eqref{Equ:ContingencyAcceleration}, which have been used to set up the collision avoidance constraints. However, the optimal contingency trajectories
	\begin{equation}
		\mathbf{\tilde{s}}_{1|0}^{(j)\star},\dots,\mathbf{\tilde{s}}_{N|0}^{(j)\star}\qquad\text{and}\qquad\mathbf{\tilde{s}}_{1|0}^{(m)\star},\dots,\mathbf{\tilde{s}}_{N|0}^{(m)\star}
	\end{equation}
	are identical to the contingency trajectories that are used to set up the new collision avoidance constraints in time step $k=1$, based on the solution shift \eqref{Equ:ShiftedInputSeq}. They satisfy the respective collision avoidance constraints, because they are part of the optimal solution of the two FTOCPs \eqref{Equ:FTOCP}. Since the collision avoidance constraints are pairwise parallel for $j$ and $m$, as illustrated in Figure \ref{Fig:ProofNoCollision}, the distances between the contingency states is upper bounded by
	\begin{equation}
		d_{i|0}^{(j,m)\star}\eqdef\left\|\mathbf{\tilde{p}}_{i|0}^{(j)\star}-\mathbf{\tilde{p}}_{i|0}^{(m)\star}\right\|\geq2\rho\quad\;\fa i=1,\dots,\tilde{N}\ef
	\end{equation}
	Therefore it is possible to set up new collision avoidance constraints between agents $j$ and $m$ in time step $k=1$. Moreover, by construction, these collision avoidance are satisfied by the candidate contingency state sequence \eqref{Equ:ShiftedStateSeq}.
	
	This concludes the proof, except for the special case where $\tilde{N}_{1}^{(m)}=0$. In this case, the agent is at rest in step $k=1$, because of \eqref{Equ:ContingencyHorizon}. The canoncially feasible case of the CMC algorithm is now (3.b), instead of (3.c), which does not exist. The contingency plan for agent $m$ is to remain at rest, due to (\ref{Equ:FTOCP}i). The rest of the proof carries over identically.
\end{proof}

\begin{remark}[Plug \& play]\label{The:PlugNPlay}
	(a) The contingency plan of each agent, as defined in Section \ref{Sec:CoCoDefinition}, only depends on its current state $\mathbf{s}_{k}^{(m)}$; more precisely, its velocity $\mathbf{v}_{k}^{(m)}$. (b) Therefore the CMC approach is naturally plug \& play. A new agent can be introduced to (or removed from) the swarm dynamically, in the following sense. If a new agent does not affect the feasibility of all existing agents upon its entry into the swarm, then the recursive feasibility of the entire swarm will be maintained. 
\end{remark}

\noindent Remark \ref{The:PlugNPlay}(b) follows immediately from Theorem \ref{The:RecFeasibility}.

\section{NUMERICAL EXAMPLES}\label{Sec:Example}

The CMC algorithm is illustrated in two numerical examples. Both examples are designed in two dimensions (coordinates $x$ and $y$), for a better graphical illustration. Apart from this, the implementation is exactly as described in the paper. 

In each example, the swarm includes five agents\linebreak $m=1,2,3,4,5$, each represented by an individual color (red, blue, green, purple, yellow). The mission for each agent $m$ is to reach a given target point $\mathbf{p}_{\mathrm{ref}}^{(m)}$. The cost function \eqref{Equ:CostFunction} is chosen in the quadratic form
\begin{subequations}
\begin{multline}\label{Equ:ExCostFunction}
	J^{(m)}\left(\cdot\right)=\sum_{i=0}^{N-1}\mathbf{a}_{i|k}^{(m)\mathrm{T}}\mathbf{R}\mathbf{a}_{i|k}^{(m)}+\mathbf{v}_{N|k}^{(m)\mathrm{T}}\mathbf{Q}\mathbf{v}_{N|k}^{(m)}+\\
	+\left(\mathbf{p}_{N|k}^{(m)}-\mathbf{p}_{\mathrm{ref}}^{(m)}\right)^{\mathrm{T}}\mathbf{S}\left(\mathbf{p}_{N|k}^{(m)}-\mathbf{p}_{\mathrm{ref}}^{(m)}\right)\ec
\end{multline}
where the weighting matrices are diagonal,
\begin{equation}\label{Equ:ExCostMatrices}
	\mathbf{R}=
	\begin{bmatrix}
	  R\;\;0\\ 0\;\;R
	\end{bmatrix}\ec\quad
	\mathbf{Q}=
	\begin{bmatrix}
	  Q\;\;0\\ 0\;\;Q
	\end{bmatrix}\ec\quad
	\mathbf{S}=
	\begin{bmatrix}
	  S\;\;0\\ 0\;\;S
	\end{bmatrix}\ef
\end{equation}
\end{subequations}
A full list of selected parameters is provided in Table \ref{Tab:NumExVariables}. The source code for all examples is available here:
\begin{center}
	\small\texttt{https://github.com/GSchildbach/CMC}
\end{center}

\renewcommand{\arraystretch}{1.5}
\begin{table}[H]
	\centering
	\begin{tabular}{c|c}
	\textbf{variable} & \textbf{value}\\\hline\hline
	$M$ & $5$\\\hline
	$\rho$ & $1\,\mathrm{m}$\\\hline
	$\Delta t$ & $0.2\,\mathrm{s}$\\\hline
	$\bar{v}$ & $3\,\frac{\mathrm{m}}{\mathrm{s}}$\\\hline
	$\bar{a}$ & $3\,\frac{\mathrm{m}}{\mathrm{s}^2}$
	\end{tabular}\hspace*{1cm}
	\begin{tabular}{c|c}
	\textbf{variable} & \textbf{value}\\\hline\hline
	$N$ & $12$\\\hline
	$\tilde{N}$ & $5$\\\hline
	$R$ & $1$\\\hline
	$S$ & $20$\\\hline
	$Q$ & $2$
	\end{tabular}
	\caption{List of parameters for all simulations.\label{Tab:NumExVariables}}
\end{table}
\renewcommand{\arraystretch}{1.0}

\subsection{Randomized Target Points (RTP)}

In the Randomized Target Points (RTP) example, all agents operate jointly on a rectangular $20\,\mathrm{m}\times 20\,\mathrm{m}$ area, corresponding to the position constraint set $\mathbb{S}^{(m)}$. The agents are spawned uniformly at random, and so are the sets of target points. Each set of target points is valid for $20\,\mathrm{s}$. After that, a new set of target points is assigned to the agents. Four sets of target points are generated in total. This leads to a simulation time of $80\,\mathrm{s}$.

Plots of the agent trajectories can be found in Figure \ref{Fig:RTPTrajectories}. For a better illustration, a video of the RTP example is available under the following link:
\begin{center}
	\small\texttt{https://youtube.com/shorts/VouqKlKyhNM}
\end{center}

\begin{figure}[H]
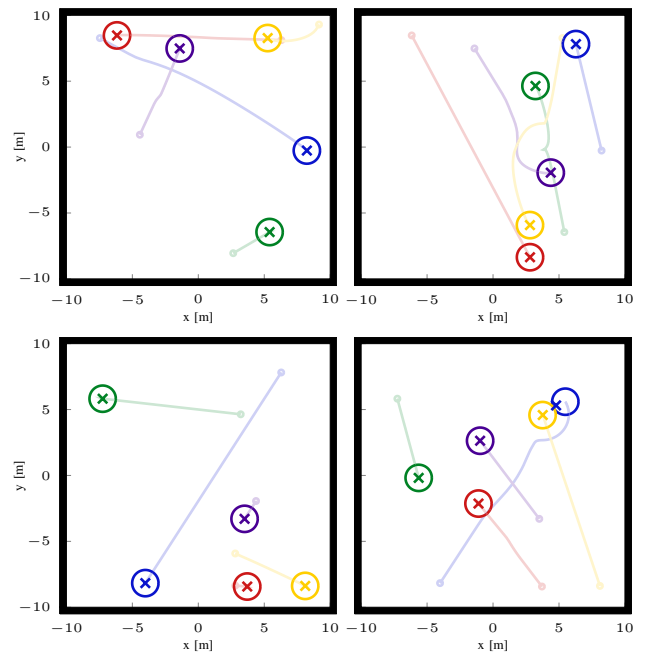

	\input{./fig/expRTP/fig01.tex}\hspace*{0.7cm}
	\input{./fig/expRTP/fig02.tex}\vspace*{-12.1cm}\\
	\input{./fig/expRTP/fig03.tex}\hspace*{0.7cm}
	\input{./fig/expRTP/fig04.tex}\vspace*{-12.3cm}
	\caption{RTP example: Overview of agent trajectories. The timeline of the plots is top left ($0\,\mathrm{s}\;\text{to}\;20\,\mathrm{s}$), top right ($20\,\mathrm{s}\;\text{to}\;40\,\mathrm{s}$), bottom left ($40\,\mathrm{s}\;\text{to}\;60\,\mathrm{s}$), bottom right ($60\,\mathrm{s}\;\text{to}\;80\,\mathrm{s}$). Agents are shown as dark circles, target points as dark crosses, trajectories as light lines, and initial conditions as small light circles.\label{Fig:RTPTrajectories}}
\end{figure}

In the last plot on the bottom right, the blue agent does not completely reach its target point. A collision avoidance constraint prevents a closer tracking, since the distance between the target points of the blue and yellow agents is less than $2\rho$. The agent who reaches its target point earlier (here: the yellow agent) has the smaller tracking error, because there is no communication.

Closed-loop collision avoidance is evaluated in Figure \ref{Fig:RTPDistances}. It shows the pairwise distance between all agents over the total simulation time. Note that an overlap between the agents is not prevented physically, by the simulation model. Instead, collisions (i.e., overlaps between the agents) are avoided if and only if the agents maintain a mutual distance of at least $2\rho$ at all times, based on the CMC algorithm.

\begin{figure}[t]
	\centering
    \includegraphics[scale=1]{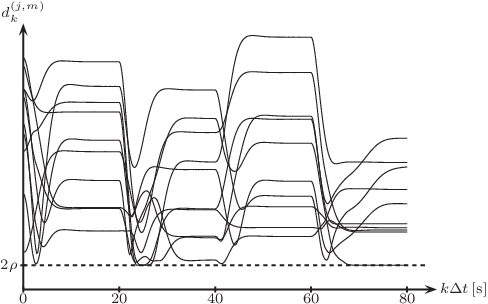} \vspace*{-0.1cm}	
    \caption{RTP example: mutual distances between all pairs of agents $m$ and $j$.\label{Fig:RTPDistances}}
\end{figure}

\subsection{Bottleneck (BTN)}

In the Randomized Target Points (RTP) example, all agents operate jointly on a rectangular $60\,\mathrm{m}\times 20\,\mathrm{m}$ area. A bottleneck separates the random initial conditions on the left and the target positions, in reverse order, on the right. The rectangular boundaries and the bottleneck represent the position constraint set $\mathbb{S}^{(m)}$. The total simulation time is $30\,\mathrm{s}$.

Plots of the agent trajectories can be found in Figure \ref{Fig:BTNTrajectories}. For a better illustration, a video of the BTN example is available under the following link:
\begin{center}
	\small\texttt{https://youtube.com/shorts/cMxOp4Oldgg}
\end{center}

\begin{figure}[H]
	\vspace*{-7.2cm}
	\input{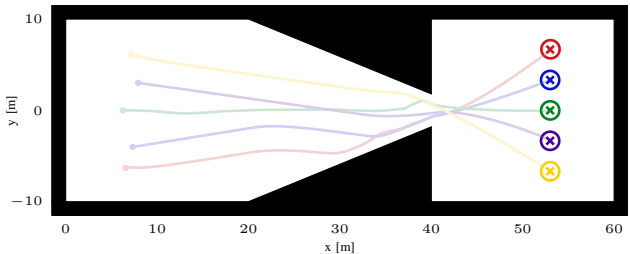}\vspace*{-4.4cm}
	\caption{BTN example: Overview of agent trajectories. Agents are shown as dark circles, target points as dark crosses, trajectories as light lines, and initial conditions as small light circles.\label{Fig:BTNTrajectories}}
\end{figure}

Closed-loop collision avoidance is evaluated in Figure \ref{Fig:BTNDistances}, showing again the pairwise distance between all agents over the total simulation time. Again, the agents manage to maintain a mutual distance of at least $2\rho$ at all times, based on the CMC algorithm.

\begin{figure}[b]
	\centering
    \includegraphics[scale=1]{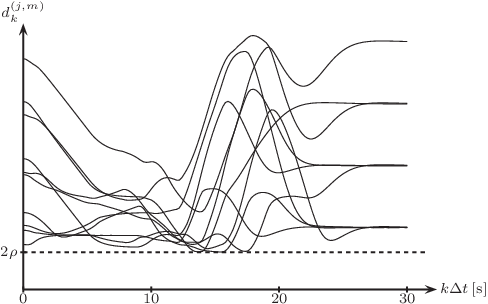} \vspace*{-0.1cm}	
    \caption{BTN example: mutual distances between all pairs of agents $m$ and $j$.\label{Fig:BTNDistances}}
\end{figure}

\subsection{Discussion}

The CMC algorithm operates smoothly and leads to an overall swarm behavior that appears reasonable and intuitive. The decentralized collision avoidance works without any observable flaws. In particular, the agent maneuvers do not seem overly conservative. As expected, no feasibility issues have occurred. The cost function has been tuned in an ad hoc fashion. It could be modified for an optimized tracking. Further experiments may also include agents with completely different objectives, i.e., other than reaching target points.

Occasionally, in both examples, deadlocks have occurred. This means the progress of an agent is blocked by other agent(s), possibly together with the state constraints. Deadlocks are a common phenomenon in decentralized approaches. They are promoted by certain obstacle geometries and symmetries in the problem. This why the initial positions in the BTN example have been randomized. Deadlocks may be reduced, or resolved, by additional rules that may be included in the CMC algorithm.

\section{CONCLUSION}

Contingency Model-based Control (CMC) is an approach for cooperative collision avoidance among a swarm of robots. Instead of communication, it is based on set of consensual rules among the robots, where a contingency trajectory is used to construct collision avoidance constraints. Formal guarantees of recursive feasibility and collision avoidance in closed-loop operation can be established. In practice, the CMC algorithm appears to operate smoothly, without excessive conservatism, in two small numerical examples.

In this paper, the CMC algorithm is presented in its most elementary form. Many questions related to CMC remain to be addressed future research. In particular, the algorithm should also work for systems with general linear, and possibly even nonlinear, dynamics. It should also extend to the case heterogenous agents, at least if the agent type can be detected by onboard sensors. Future work may also explore the integration of measurement noise, process noise, model uncertainty and non-synchronized clocks. Modifications of the CMC algorithm may be considered in order to reduce, or resolve, deadlocks---or, more generally, to improve its closed-loop behavior. As part of this, also the use of alternative contingency strategies may be explored.


\bibliographystyle{IEEEtran}
\bibliography{bibmath,bibcontr,bibeng,bibecon}

\begin{thebibliography}{10}
\providecommand{\url}[1]{#1}
\csname url@rmstyle\endcsname
\providecommand{\newblock}{\relax}
\providecommand{\bibinfo}[2]{#2}
\providecommand\BIBentrySTDinterwordspacing{\spaceskip=0pt\relax}
\providecommand\BIBentryALTinterwordstretchfactor{4}
\providecommand\BIBentryALTinterwordspacing{\spaceskip=\fontdimen2\font plus
\BIBentryALTinterwordstretchfactor\fontdimen3\font minus
  \fontdimen4\font\relax}
\providecommand\BIBforeignlanguage[2]{{%
\expandafter\ifx\csname l@#1\endcsname\relax
\typeout{** WARNING: IEEEtran.bst: No hyphenation pattern has been}%
\typeout{** loaded for the language `#1'. Using the pattern for}%
\typeout{** the default language instead.}%
\else
\language=\csname l@#1\endcsname
\fi
#2}}

\bibitem{Hamann:2018}
H.~Hamann, \emph{Swarm Robotics: A Formal Approach}, 1st~ed.\hskip 1em plus
  0.5em minus 0.4em\relax Cham, Switzerland: Springer, 2018.

\bibitem{SvestkaOvermars:1998}
P.~{\v{S}}vestka and M.~H. Overmars, ``Coordinated path planning for multiple
  robots,'' \emph{Robotics and Autonomous Systems}, vol.~23, no.~3, pp.
  125--152, 1998.

\bibitem{PreissEtAl:2017}
J.~A. Preiss, W.~H{\"o}nig, N.~Ayanian, and G.~S. Sukhatme, ``Downwash-aware
  trajectory planning for large quadrotor teams,'' in \emph{IEEE/RSJ
  International Conference on Intelligent Robots and Systems (IROS)}, Vancouver
  (BC), Canada, 2017, pp. 250--257.

\bibitem{SchouwenaarsEtAl:2001}
T.~Schouwenaars, B.~{De Moor}, E.~Feron, and J.~How, ``Mixed {I}nteger
  {P}rogramming for multi-vehicle path planning,'' in \emph{European Control
  Conference (ECC)}, Porto, Portugal, 2001, pp. 2603--2608.

\bibitem{AugugliaroEtAl:2012}
F.~Augugliaro, A.~P. Schoellig, and R.~D'Andrea, ``Generation of collision-free
  trajectories for a quadrocopter fleet: A sequential convex programming
  approach,'' in \emph{IEEE/RSJ International Conference on Intelligent Robots
  and Systems (IROS)}, Vilamoura-Algarve, Portugal, 2012, pp. 1917--1922.

\bibitem{StipanovicEtAl:2007}
D.~M. Stipanovi{\'c}, P.~F. Hokayem, M.~W. Spong, and D.~D. {\v{S}}iljak,
  ``Cooperative avoidance control for multiagent systems,'' \emph{Journal of
  Dynamic Systems, Measurement, and Control}, vol. 129, no.~5, pp. 699--707,
  2007.

\bibitem{VanParysPipeleers:2017}
R.~{Van Parys} and G.~Pipeleers, ``Distributed {M}odel {P}redictive {F}ormation
  {C}ontrol with inter-vehicle collision avoidance,'' in \emph{Asian Control
  Conference (ASCC)}, Gold Coast (QLD), Australia, 2017, pp. 2399--2404.

\bibitem{ChalousEtAl:2010}
G.~Chaloulos, P.~Hokayem, and J.~Lygeros, ``Distributed hierarchical {MPC} for
  conflict resolution in air traffic control,'' in \emph{American Control
  Conference (ACC)}, Baltimore (MD), United States, 2010, pp. 3945--3950.

\bibitem{ChengEtAl:2017}
H.~Cheng, Q.~Zhu, Z.~Liu, T.~Xu, and L.~Lin, ``Decentralized navigation of
  multiple agents based on {ORCA} and {M}odel {P}redictive {C}ontrol,'' in
  \emph{IEEE/RSJ International Conference on Intelligent Robots and Systems
  (IROS)}, Vancouver (BC), Canada, 2017, pp. 3446--3451.

\bibitem{LuisSchoellig:2019}
C.~E. Luis and A.~P. Schoellig, ``Trajectory generation for multiagent
  point-to-point transitions via {D}istributed {M}odel {P}redictive
  {C}ontrol,'' \emph{IEEE Robotics and Automation Letters}, vol.~4, no.~2, pp.
  375--382, 2019.

\bibitem{GraefeEtAl:2025}
A.~Gr{\"a}fe, J.~Eickhoff, M.~Zimmerling, and S.~Trimpe, ``{DMPC-Swarm}:
  {D}istributed {M}odel {P}redictive {C}ontrol on nano {UAV} swarms,''
  \emph{Autonomous Robots}, vol.~49, no.~28, pp. 1--19, 2025.

\bibitem{OngGerdes:2015}
H.~Y. Ong and J.~C. Gerdes, ``Cooperative collision avoidance via proximal
  message passing,'' in \emph{American Control Conference (ACC)}, Chicago (IL),
  United States, 2015, pp. 124--4130.

\bibitem{SaccaniFagiano:2021}
D.~Saccani and L.~Fagiano, ``Autonomous {UAV} navigation in an unknown
  environment via {M}ulti-trajectory {M}odel {P}redictive {C}ontrol,'' in
  \emph{European Control Conference (ECC)}, Delft, Netherlands, 2021, pp.
  1577--1582.

\bibitem{CarronEtAl:2023}
A.~Carron, D.~Saccani, L.~Fagiano, and M.~N. Zeilinger, ``Multi-agent
  {D}istributed {M}odel {P}redictive {C}ontrol with connectivity constraint,''
  \emph{IFAC-PapersOnLine}, vol.~56, no.~2, 2023.

\bibitem{BhattacharyaEtAl:2011}
S.~Bhattacharya, V.~Kumar, and M.~Likhachev, ``Distributed optimization with
  pairwise constraints and its application to multi-robot path planning,'' in
  \emph{Robotics: Science and Systems {VI}}, J.~A. Cetto, J.-L. Ferrier,
  J.~M.~C. dias Pereira, and J.~Filipe, Eds.\hskip 1em plus 0.5em minus
  0.4em\relax Cambridge (MA), United States: The MIT Press, 2011, pp. 177--184.

\bibitem{ToumiehFloreano:2024}
C.~Toumieh and D.~Floreano, ``High-speed motion planning for aerial swarms in
  unknown and cluttered environments,'' \emph{IEEE Transactions on Robotics},
  vol.~40, pp. 3642--3656, 2024.

\bibitem{TordesillasHow:2021}
J.~Tordesillas and J.~P. How, ``{MADER}: Trajectory planner in multiagent and
  dynamic environments,'' \emph{IEEE Transactions on Robotics}, vol.~38, no.~1,
  pp. 463--476, 2021.

\bibitem{Stoustrup:2009}
J.~Stoustrup, ``Plug \& play control: Control technology towards new
  challenges,'' \emph{European Journal of Control}, vol.~15, no.~3, pp.
  311--330, 2009.

\bibitem{AbeYoshiki:2001}
Y.~Abe and M.~Yoshiki, ``Collision avoidance method for multiple autonomous
  mobile agents by implicit cooperation,'' in \emph{IEEE/RSJ International
  Conference on Intelligent Robots and Systems (IROS)}, Maui (HI), United
  States, 2001, pp. 1207--1212.

\bibitem{ChangEtAl:2003}
D.~E. Chang, S.~C. Shadden, J.~E. Marsden, and R.~Olfati-Saber, ``Collision
  avoidance for multiple agent systems,'' in \emph{42nd IEEE International
  Conference on Decision and Control (CDC)}, Maui (HI), United States, 2003,
  pp. 539--543.

\bibitem{RezaeeAbdollahi:2014}
H.~Rezaee and F.~Abdollahi, ``A decentralized cooperative control scheme with
  obstacle avoidance for a team of mobile robots,'' \emph{IEEE Transactions on
  Industrial Electronics}, vol.~61, no.~1, pp. 347--354, 2014.

\bibitem{ZhuEtAl:2020}
E.~L. Zhu, Y.~R. St{\"u}rz, U.~Rosolia, and F.~Borrelli, ``Trajectory
  optimization for nonlinear multi-agent systems using {D}ecentralized
  {L}earning {M}odel {P}redictive {C}ontrol,'' in \emph{59th IEEE Conference on
  Decision and Control (CDC)}, Jeju, South Korea, 2020, pp. 6198--6203.

\bibitem{FioriniShiller:1998}
P.~Fiorini and Z.~Shiller, ``Motion planning in dynamic environments using
  velocity obstacles,'' \emph{The International Journal of Robotics Research},
  vol.~17, pp. 760--772, 2022.

\bibitem{VanDenBergEtAl:2009}
J.~{Van Den Berg}, S.~J. Guy, M.~Lin, and D.~Manocha, ``Reciprocal n-body
  collision avoidance,'' in \emph{Robotics Research}, C.~Pradalier,
  R.~Siegwart, and G.~Hirzinger, Eds.\hskip 1em plus 0.5em minus 0.4em\relax
  Berlin, Heidelberg, Germany: Springer, 2009, pp. 3--19.

\bibitem{VanDenBergEtAl:2011}
J.~{Van Den Berg}, J.~Snape, S.~J. Guy, and D.~Manocha, ``Reciprocal collision
  avoidance with acceleration-velocity obstacles,'' in \emph{IEEE International
  Conference on Robotics and Automation (ICRA)}, Shanghai, China, 2011, pp.
  3475--3482.

\bibitem{AlonsoMoraEtAl:2013}
J.~{Alonso-Mora}, A.~Breitenmoser, M.~Rufli, P.~Beardsley, and R.~Siegwart,
  ``Optimal reciprocal collision avoidance for multiple non-holonomic robots,''
  in \emph{{D}istributed {A}utonomous {R}obotic {S}ystems: {T}he 10th
  {I}nternational {S}ymposium}, A.~Martinoli and {et al.}, Eds.\hskip 1em plus
  0.5em minus 0.4em\relax Heidelberg Berlin, Germany: Springer, 2013, pp.
  203--216.

\bibitem{KratkyEtAl:2025}
\BIBentryALTinterwordspacing
V.~Kr{\'a}tk{\'y}, R.~P{\v e}ni{\v c}ka, P.~M. Gupta, O.~Proch{\'a}zka, and
  M.~Saska, ``{RVC-NMPC}: {N}onlinear {M}odel {P}redictive {C}ontrol with
  reciprocal velocity constraints for mutual collision avoidance in agile {UAV}
  flight,'' 2025. [Online]. Available: \url{https://arxiv.org/pdf/2512.08574}
\BIBentrySTDinterwordspacing

\end{thebibliography}


\end{document}